\documentclass[11pt]{amsart}

\usepackage{amscd,amssymb,amsopn,amsmath,amsthm,mathrsfs,graphics,amsfonts,enumerate,verbatim,calc,enumerate}
\usepackage[colorlinks=true, urlcolor=blue, linkcolor=red,citecolor=purple]{hyperref}
\usepackage[all]{xy}
\usepackage[shortlabels]{enumitem}
\usepackage{url}
\newlist{steps}{enumerate}{1}
\usepackage{xcolor}

\usepackage[OT2,OT1]{fontenc}
\newcommand\cyr{%
\renewcommand\rmdefault{wncyr}%
\renewcommand\sfdefault{wncyss}%
\renewcommand\encodingdefault{OT2}%
\normalfont
\selectfont}
\DeclareTextFontCommand{\textcyr}{\cyr} 

\usepackage{amssymb,amsmath}

\DeclareFontFamily{OT1}{rsfs}{}
\DeclareFontShape{OT1}{rsfs}{n}{it}{<-> rsfs10}{}
\DeclareMathAlphabet{\mathscr}{OT1}{rsfs}{n}{it}

\numberwithin{equation}{section}
\newtheorem{theorem}{Theorem}[section]
\newtheorem{lemma}[theorem]{Lemma}
\newtheorem{proposition}[theorem]{Proposition}
\newtheorem{corollary}[theorem]{Corollary}

\theoremstyle{definition}

\newtheorem{remark}[theorem]{Remark}

\theoremstyle{remark}

\newcommand{\Spec}{\operatorname{Spec}}

\newcommand{\depth}{\operatorname{depth}}

\newcommand{\Frac}{\operatorname{Frac}}

\newcommand{\Proj}{\operatorname{Proj}}

\newcommand{\Max}{\operatorname{Max}}

\newcommand{\fm}{\mathfrak{m}}

\newcommand{\fa}{\mathfrak{a}}

\begin{document}
\title[Quasi-Gorenstein (normal)  finite covers in arbitrary characteristic]
{Quasi-Gorenstein (normal) finite covers in arbitrary characteristic}

\author[E. Tavanfar]{Ehsan Tavanfar}
\address{School of Mathematics, Institute for Research in Fundamental Sciences (IPM), P. O. Box: 19395-5746, Tehran, Iran}
\email{tavanfar@ipm.ir}

\thanks{
This research   was supported by a grant from IPM}



\begin{abstract} 
	   We show that any complete local (normal) domain  admits a module-finite  quasi-Gorenstein normal (complete local) domain extension. In the geometric vein, we show that any normal projective variety $X$ over a field admits a finite surjective morphism $Y\rightarrow X$ from a normal quasi-Gorenstein projective variety $Y$.   Notably, our results  resolve   the previously open case for residual characteristic two.
\end{abstract}

\maketitle
\tableofcontents

\section{Introduction}

   When working with singularities in algebraic geometry and commutative algebra, it is often desirable to obtain a  ``better" scheme that retains key properties of the original but is less singular or more regular. Grothendieck's conjecture on the existence of resolutions of singularities provides one such approach, arguably  the most important, and it has been established for projective varieties containing $\mathbb{Q}$ and for threefolds. By Zariski's main theorem,   resolutions of normal singularities are never finite, so-called a \textit{finite cover}. Here, by a \textit{finite cover} $Y$ of a normal scheme $X$,   a finite surjective and ``normal" scheme $Y$ over  $X$ is meant. If    $Y$ is additionally  quasi-Gorenstein (i.e. if the canonical divisor of $Y$ is a Cartier divisor), then we say that $Y$ is a \textit{quasi-Gorenstein finite cover} of $X$. 
   
   A natural question arises: what kind of a  better finite cover of a normal singular scheme exists?    Understanding ``better finite covers"  perhaps may also help in understanding regular alterations because an alteration of a finite cover of a projective variety $X$ is an alteration of $X$. One nice example of this type is that a strictly henselian local Gorenstein isolated singularity $(R,\fm)$ of dimension $2$ (so-called a rational double point) admits a non-singular finite cover (\cite{PostralEtALCovers}). 
   
   In this paper we address this question by proving the existence of a  quasi-Gorenstein finite cover $Y$ (over $X$), provided either $X$ is affine and is the spectrum of   a (normal) complete local domain, or $X$ is a normal projective variety over a field.
 
   Particularly, our  result  resolves, in the aforementioned cases, the previously open case of the existence of quasi-Gorenstein finite covers in  residual characteristic two.  Even though the quest for such finite covers does not seem to be questioned precisely in the literature, but   a rich history for it can be found in the literature.
   
   Most notably, a construction of  Yujiro Kawamata  in  \cite{KawamataCrepant} establishes the existence of quasi-Gorenstein finite covers for a very wide class of schemes with no point of residual characteristic $2$. Namely, the most general version of Kawamata's construction is the following theorem pointed out to the author by Karl Schwede who   learned the existence of such covers  from J\'anos Koll\'ar (in residual characteristic $\neq 2$) in a short conversation in 2012:

\begin{theorem}(mainly due to Y. Kawamata, this stated version  by Koll\'ar-Schwede)\label{TheoremKawamta}
  Suppose $X$ is a normal integral Noetherian scheme where no point has residue field characteristic equal to 2.  Suppose further there exists a Cartier divisor $H$, a canonical divisor $K_X$ and that some section $f \in \Gamma(X, \mathcal{O}_X(-2K_X + 2H))$ determines a reduced divisor $D \sim -2K_X + 2H$.  Then the associated ramified cyclic cover:
  \[
    Y = \Spec\big( \mathcal{O}_X \oplus \mathcal{O}_X(K_X-H) \big)
  \]
  is normal and quasi-Gorenstein.

If $X$ is a variety over an infinite field of characteristic $\neq 2$, it is easy to see that such an $f$ exists by Bertini's theorem when $H$ is sufficiently ample.

\end{theorem}

  Here, Bertini theorem in the projective variety case holds by \cite[Bertini's Theorem 3.4.8]{FlennerOcarrolVogelJoins}.  The multiplication of $\mathcal{O}_Y$ in the statement is given by  (on an affine open subset $U$) $(0,x)(0,y):=(gxy,0)$   where  the fixed rational function $0\neq g\in K(X)$ is such that $\text{div}(g)=D+2K_X-2H$ and $x,y\in \big(\mathcal{O}_X(K_X-H)\big)(U)$. The essentially of finite type over a field of characteristic $\neq 2$ version of  Theorem \ref{TheoremKawamta} (without mentioning Bertini theorem in the statement) can be found in \cite[Lemma 3.5]{McDonaldMultiplier}. For the  proof of Theorem \ref{TheoremKawamta} an interested reader is referred to \cite{KawamataCrepant} and \cite[Lemma 3.5]{McDonaldMultiplier}. In the final section of this paper we show that a version of Theorem \ref{TheoremKawamta} in which the canonical divisor $K_X$ is replaced by an arbitrary Weil divisor, settles the complete local case of the main result of \cite{PatankarVanishing} concerning the regularity of a normal local domain whose absolute integral closure and residue field are tor-independent for some positive $i$.

A different approach for obtaining quasi-Gorenstein normal covers was established by Phillip Griffith in \cite{GriffithSomeResults} by finding  firstly a $\mathbb{Q}$-Gorenstein finite cover $S$ (by taking a normal extension of a Noether normalization) and then considering the canonical cover of $S$ which is known to be quasi-Gorenstein and normal provided the order of the canonical divisor of $S$ is relatively prime to the characteristic of the residue field. 

In this paper, our approach is to  complete  Griffith's approach by finding a quasi-Gorenstein finite cover of such an $\mathbb{Q}$-Gorenstein finite cover $S$, independent of the order of the canonical divisor of $S$ (thus we do not take the cyclic (canonical) cover). The common trick in our existence result and in that of Theorem \ref{TheoremKawamta} is that both are somehow applying Bertini theorem.

\section{Quasi-Gorenstein finite cover: the complete local case}

\begin{proposition}
    \label{PropositionCompleteInfiniteResidueField} Let $(R,\mathfrak{m})$
	be an excellent henselian local normal domain with an infinite residue field. Assume that at least one of the following conditions hold:
 \begin{enumerate}[(i)]\label{EnumerateTwoCasesCompleteOrGoodHenselian}
   \item\label{itm:1Complete} $R$ is complete.
   \item \label{itm:2Henselian} $R/\mathfrak{m}$ is algebraically closed and $R$ admits a separable Noether normalization. 
 \end{enumerate}  
   Then $(R,\mathfrak{m})$	admits a  generically separable quasi-Gorenstein finite cover $S$.
\end{proposition}
\begin{proof}
	If $\dim(R)\le 1$, then $R$ is regular and there is nothing to prove. So we assume that $\dim(R)\ge 2$. 
 
		By our hypothesis, there exists a Noether	normalization $\sigma:P\rightarrow R$ which is  generically separable, i.e. $P$ is a  regular local ring, the extension $\sigma$ is module finite and induces an isomorphism on the residue fields and $\text{Frac}(P)\rightarrow\text{Frac}(R)$ is a separable finite field extension (here we stress that for the complete case in  \ref{itm:1Complete},    this is a well known fact that a Noether normalization exists, whose generic separability is immediate in characteristic zero and holds in prime characteristic  by virtue of \cite{KuranoShimomoto}). 
  
  In particular,  the extension $\text{Frac}(P)\rightarrow\text{Frac}(R)$  is primitive and $\text{Frac}(R)=\text{Frac}(P)[\theta]$	for some separable element $\theta\in\text{Frac}(R)$. So the splitting field, $L$, of the minimal polynomial of $\theta$  over $\text{Frac}(P)$ is a finite separable normal extension, i.e. a Galois extension, of $\text{Frac}(P)$. 	Let $R'$ be the integral closure of $R$ (equivalently, $P$) in	$L$.	
  
  The domain  $R'$ is a  module finite extension of $R$ (\cite[Lemma 1, page 262]{Matsumura}, or \cite[Theorem 4.3.4]{HunekeSwansonIntegral} for only the complete case), thus it is  an excellent henselian (\cite[Tag 04GH]{StacksProject}) local normal domain  extension of $R$ which admits a Noether normalization (in case \ref{itm:1Complete} $R'$ is complete and thus it admits a Noether normalization, and in case \ref{itm:2Henselian} $P\rightarrow R'$ is a Noether normalization of $R'$ too by the algebraically closedness of $R/\mathfrak{m}$). In particular, as a generically Galois extension of $P$,  $R'$ is $\mathbb{Q}$-Gorenstein  in light of  and art of \cite[Theorem 1.3(2) and Corollary 2.4]{KuranoOnRoberts}, or more originally, in light of  and art of \cite[Theorem 2.1]{GriffithPathology} \footnote{About applying the second reference \cite[Theorem 2.1]{GriffithPathology}, this is the case as 
	the canonical ideal is fixed under any ring automorphism up to isomorphism,
	therefore the canonical class is fixed under any ring automorphism. Here, we should also remark that in the statement of \cite[Theorem 2.1]{GriffithPathology} it is assumed for the extension to be a \textit{normal extension} and in the given definition of a \textit{normal extension} in \cite[Page 37]{GriffithPathology} the author requires further that the order of the Galois group $G$ is invertible. However, chasing the proof of \cite[Theorem 2.1]{GriffithPathology} we can see that the hypothesis on the order of the Galois group being invertible is not used in the proof and therefore we are safe to apply  \cite[Theorem 2.1]{GriffithPathology}. Anyways, the first reference \cite[Theorem 1.3(2) and Corollary 2.4]{KuranoOnRoberts} does not require any such superfluous condition like invertibility of the order of a group.}.
	
	 Thus to complete the proof we may, and we do, assume furthermore that $R$ is    $\mathbb{Q}$-Gorenstein, say of index $n$, implying for some $f\in \Frac(R)$ that 
	 \begin{equation}
	   \label{EquationQGorensteinImpliesFiniteGeneratation}	
	   \mathscr{R}[\omega_R^{(-1)}t]_{kn+i}=f^k\mathscr{R}[\omega_R^{(-1)}t]_{i}, \ \ \forall\ 0\le i\le n-1,\ k\in \mathbb{N}, 
	 \end{equation}
	  where  $$\mathscr{R}[\omega_R^{(-1)}t]:=\oplus_{j\in \mathbb{N}_0} \omega_R^{(-j)}t^j$$  is  the \textit{anticanonical algebra} of $R$. It turns out that then  $\mathscr{R}[\omega_R^{(-1)}t]$ is a finitely generated  algebra over $R$ and therefore  $\mathscr{R}[\omega_R^{(-1)}t]$  (and thus its completion $\widehat{\mathscr{R}[\omega_R^{(-1)}t]}$ as well) is a Noetherian  quasi-Gorenstein normal domain of depth $\ge 3$, by \cite[Lemma (4.3)(1) and Theorem 4.5]{GotoEtALOnTheStructure}.
	 We may only need to clarify the mentioned depth lower bound, which is essentially required when applying  Bertini  type theorems  later on: Since $R$ is $\mathbb{Q}$-Gorenstein,  from display (\ref{EquationQGorensteinImpliesFiniteGeneratation}) we get the  quotient
	 $$\mathscr{S}:=\mathscr{R}[\omega_R^{(-1)}t]/ft^n=R\oplus \omega_R^{(-1)}t\oplus \cdots\oplus \omega_R^{(-n+1)}t^{n-1},\ \    (\omega_R^{(-n)}=fR)$$  showing  that $\depth(\mathscr{S})$ as a local  ring, equivalently as a finite $R$-module, is clearly bounded by $2$, i.e.  $\depth(\mathscr{R}[\omega_R^{(-1)}t])\ge 3$ as mentioned.
	 
	  Keeping the assumption  $\omega_R^{(-n)}=fR$,   let the finite set $\{w_{j,i}\}$ forms a minimal generating set of $\omega_R^{(-j)}$ for each $1\le j\le n-1$, and  $\mathfrak{m}=(m_1,\ldots,m_\mu)$. Hence,   $\{ft^n,\omega_{j,i}t^j,m_i\}$ generates the unique homogeneous maximal ideal of $\mathscr{R}[\omega_R^{(-1)}t]$. 
	   Then,  applying \cite[(4.2) Korollar]{FlennerDieSatze} (local Bertini theorem for the equal characteristic case) and \cite[Theorem (4.4) (Local Bertini theorem)]{OchiaiShimomotoBertini} (for the mixed characteristic case) to the completion, $\widehat{\mathscr{R}[\omega_R^{(-1)}t]}$, of $\mathscr{R}[\omega_R^{(-1)}t]$   we can find  a general choice of an element $z$ of the form 
	     \begin{equation}
	     	\label{EquationTheBertinieElement}
	       z=ft^n+(\sum_i k_{n-1,i}w_{n-1,i})t^{n-1}+\cdots + (\sum_i k_{1,i}w_{1,i})t + (\sum_i k_{0,i} m_i)
	     \end{equation}  
	   of $\mathscr{R}[\omega_R^{(-1)}t]$ such that it fulfills the following three properties:
	  
	  \begin{enumerate}[(i)]\label{EnumerateGeneralZProperties}
	   \item \label{itm:Normal} $\mathscr{R}[\omega_R^{(-1)}t]/(z)$ is a normal local domain, where the non-zero coefficients $k_{j,i}$'s in display (\ref{EquationTheBertinieElement}) are invertible elements of a prefixed coefficient ring of $\mathscr{R}[\omega_R^{(-1)}t]$ (that agrees with a prefixed coefficient ring of $R$ \footnote{The ring $R$ admits a Noether normalization as discussed before, thus it admits a coefficient ring as well.}). Particularly,  the normal domain $\mathscr{R}[\omega_R^{(-1)}t]/(z)$ is necessarily quasi-Gorenstein as well in view of \cite[Corollary 2.8]{TavanfarTousiAStudy}.   
	   
	   \item\label{itm:FirstIrreducibleIntersection} As $\mathscr{R}[\omega_R^{(-1)}t]$ admits a coefficient ring evidently,    applying  Bertini's theorem to $\widehat{\mathscr{R}[\omega_R^{(-1)}t]}$ actually  results in general (normal hypersurface) elements in $\mathscr{R}[\omega_R^{(-1)}t]$ rather than in its completion, therefore when we are applying the complete local version of  Bertini's theorem we can pass to the completion and  then come back safely. We stress that one important job done here is that   since two non-trivial open subsets of an irreducible space intersect non-trivially, we can indeed find such an element $z$  as mentioned in display (\ref{EquationTheBertinieElement}) whose summand $ft^n$ has coefficient $1$. 
	   
	   \item\label{itm:SecondIrreducibleIntersection} Again since the intersection of finitely many non-empty open subsets  of an irreducible space is non-empty, we  shrink  more the locus of coefficients $k_{j,i}$ of  the elements $z$ in display (\ref{EquationTheBertinieElement}) by intersecting with the non-zero locus of  the coefficient $k_{1,1}$, and we obtain some general element $z$ in display (\ref{EquationTheBertinieElement})   such that moreover $\sum_i k_{1,i}w_{1,i}\neq 0$ (because, $k_{1,1}\neq 0$ is invertible and $\{w_{1,i}\}$ is a minimal generating set of $\omega_R^{-1}$).   The conclusion of considering such a general element $z$ is an extension  $R\rightarrow S=\mathscr{R}[\omega_R^{(-1)}t]/(z)$ inducing the fraction field extension $$\Frac(R)\rightarrow \Frac(S)=\Frac(R)\otimes_RS=\Frac(R)[X]/(f(X))$$ for some  degree $n$ polynomial $f(X)$ whose monomial $X$ has a non-zero coefficient. Accordingly,  $\Frac(R)\rightarrow \Frac(S)$ is a finite separable extension as   $\partial f/\partial X\neq 0$ (\cite[Tag 09H9]{StacksProject}).		  
	\end{enumerate} 
	
	By \ref{itm:Normal},\ref{itm:SecondIrreducibleIntersection}   above  the   domain $\mathscr{R}[\omega_R^{(-1)}t]/(z)$   is quasi-Gorenstein, normal and generically separable. In view of our choice of  $z$ satisfying the condition (ii) above, it is also  readily seen that    $\mathscr{R}[\omega_R^{(-1)}t]/(z)$  is module-finite over $R$ (in view of display (\ref{EquationQGorensteinImpliesFiniteGeneratation})). So the statement follows.
\end{proof}

To extend the above result to the case of a complete local ring with a possibly finite residue field, we continue by presenting a remark followed by some lemmas.

\begin{remark}\label{RemarkFinitenessOfTheExtension}
	Let $(R,\fm)$ be a  local ring and $R'$ be a module finite extension of $R$ such that $\fm R'$ is a maximal ideal of $R'$. Then $(R',\fm R')$ is a local ring, simply because any maximal ideal of $R'$ contracts to $\fm$ hence it should contain the maximal ideal $\fm R'$.
\end{remark}

\begin{lemma}\label{LemmaStrictHenselizationAndDVRS}
	Let $V$ be a complete unramified discrete valuation ring of mixed characteristic with a perfect residue field of prime characteristic. Let $W$ be the strict henselization of $V$ (\cite[Tag 04GQ]{StacksProject}). Then $W$ is an unramified discrete valuation ring which is a  direct limit of module-finite \'etale local $V$-algebras, say $$W=\lim_{\underset{i\in I}{\longrightarrow}}V_i,$$ and the residue field of $W$ is the algebraic closure of $V$.	 
\end{lemma}
\begin{proof}
	The last sentence of the statement holds clearly in view of   \cite[Tag 04GP(5)]{StacksProject} and the perfectness of the residue field of  $V$, and $W$ is evidently an unramified discrete valuation ring in view of  \cite[Tag 04GP(1),(3) and (4)]{StacksProject}.	
	
	 The almost not trivial fact here is perhaps the expression of  $W$ as a direct limit of module finite \'etale local $V$-algebras: By \cite[Tag 04GP(2)]{StacksProject}, and by an argument as in the proof of \cite[Tag 04GN]{StacksProject} and after localizing any $R$-algebra in the direct system at an appropriate element $g$, we can view $W$ as a direct limit of local \'etale $V$-algebras, each of which residually finite separable over the residue field of $V$. Thus, since $V$ is complete, so in view of the \textit{complete version of   Nakayama's lemma}, $W$ can be described as the aforementioned desired direct limit. 
\end{proof}

\begin{lemma}\label{LemmaExcellentProperty}
	Let $(R,\mathfrak{m})$ be a complete local ring with a perfect residue field and with a coefficient ring $V$.  Let $V\rightarrow W$ be either an algebraic closure of   $V$, or the strict henselization of $V$ in the mixed  characteristic case.
	Then  $(R\otimes_V W,(\fm))$ is an excellent henselian     Noetherian local ring with an algebraically closed residue field and $\beta_W:R\rightarrow R\otimes_V W$ is a  formally \'etale local ring homomorphism. In particular,    $R$ is normal  if and only if  $R\otimes_VW$ is so.
\end{lemma}
\begin{proof} 
	In view of Lemma \ref{LemmaStrictHenselizationAndDVRS} in the mixed characteristic, and evidently in the equal characteristic, the base change $R\otimes_VW$ is a direct limit of  rings $R\otimes_VV'_i$, each of which are module finite (and \'etale, as a base change of an \'etale map) over the complete local ring $R$ and such that $(\fm)\in \Max(R\otimes_VV'_i)$, hence are local (Remark \ref{RemarkFinitenessOfTheExtension}) and are complete. So   $\big(R\otimes_VW, (\fm)\big)$  is local, Noetherian (\cite[Theorem 1]{OgomaNoetherian}), henselian (\cite[Proposition 10.1.5]{FordSeparable}) and a formally \'etale $R$-algebra (\cite[Tag 00UR and Tag 031N]{StacksProject}). In particular, $R$ is normal if and only if so is $R\otimes_VW$, as $\beta_W$ is a regular homomorphism (\cite[Tag 07PM]{StacksProject}). It is readily seen that $R\otimes_VW$ is a homomorphic image of a regular local ring and thus it universally catenary, by considering the formally \'etale (flat) local extension $P\rightarrow P\otimes_V W$ wherein the complete regular local ring $P$ is  a Cohen presentation of $R$. Consequently, $R\otimes_VW$ is excellent if and only if it is quasi-excellent.
	 Finally,  $R\otimes_VW$ is excellent  by virtue of \cite[Corollary 6.10]{MajadasCommutative}. 
    It is perhaps necessary	 to give more comments on the statement of \cite[Corollary 6.10]{MajadasCommutative}:

	The term ``\textit{smooth}'' in the statement of \cite[Corollary 6.10]{MajadasCommutative} is equivalent to the vanishing of the first Andr\'e-Quillen homology supported at the target    as well as the projectivity of the   K\"ahler differentials  (i.e. $H_1(A,B,B)=0$ and $\Omega_{B/A}$  is projective,  in the notation of \cite{MajadasCommutative}), see \cite[proof of Theorem 6.2]{MajadasCommutative} or \cite[Property 2.4]{SpivakovskyANewProof}. In particular, if $\beta_W$ is formally \'etale then
	it is a regular homomorphism of local rings  by \cite[Tag 07PM]{StacksProject}, so condition (6) in \cite[Property 2.4]{SpivakovskyANewProof} holds by \cite[Property 2.8]{SpivakovskyANewProof} and thus $\beta_W$ is smooth (in the sense of \cite{MajadasCommutative} as we need)  by \cite[Property 2.4]{SpivakovskyANewProof} because $\beta_W$ is formally \'etale and so $\Omega_{R\otimes_VW/R}=0$ (\cite[Tag 00UO]{StacksProject}) \footnote{Apparently, this notion of ``\textit{smooth}" in  \cite{MajadasCommutative} and in \cite{SpivakovskyANewProof}  is the same concept of \textit{formally smoothness} which is called as \textit{quasi-smooth} in  Richard G. Swan's  \cite[page 1]{SwanNeronPopescu}, c.f. \cite[Theorem 3.4]{SwanNeronPopescu}, and  then we immediately have that formally \'etale implies formally smooth by definition. Thus, we are allowed to apply \cite[Corollary 6.10]{MajadasCommutative}.}.    		
\end{proof}

\begin{theorem}\label{TheoremCoverCompleteCase}	Any complete local normal domain  $(R,\fm)$ admits  a generically separable quasi-Gorenstein finite cover. 
\end{theorem}
\begin{proof}
	We need only to discuss the case where $R/\fm$ is a  finite field  and $\dim(R)\ge 2$, as the infinite residue field case holds by Proposition \ref{PropositionCompleteInfiniteResidueField}.  
 
 Let $V,W$ be as in the statement of Lemma \ref{LemmaExcellentProperty}, so   $(R\otimes_VW,(\fm))$ is an excellent  henselian  Noetherian local normal domain with an algebraically closed residue field.
	
	Let $V[[\mathbf{X}]]$  be a (generically) separable Noether normalization of $R$ (again we use  \cite{KuranoShimomoto} for the prime characteristic case),  providing  us with a Noether normalization    $V[[\mathbf{X}]]\otimes_VW\rightarrow R\otimes_VW$,    of $R\otimes_VW$,  that  is also generically separable as in the case of mixed characteristic  this separability is immediate and in the prime characteristic it follows from \cite[Theorem 4.3.1, Theorem 4.5.6 and Exercise 4.4.9]{FordSeparable}: namely, from \cite[Theorem 4.3.1 and Theorem 4.5.6]{FordSeparable} and our hypothesis we get $V((\mathbf{X}))\otimes_VW\rightarrow \Frac(R)\otimes_VW$ is a separable ring homomorphism and then the desired separability follows from \cite[Exercise 4.4.9 and Theorem 4.5.6]{FordSeparable} together with the facts that $V((\mathbf{X}))\otimes_VW=\Frac(V[[\mathbf{X}]]\otimes_VW)$ and $\Frac(R\otimes_VW)$ is a localization (at the same time a quotient) of  $\Frac(R)\otimes_VW$ which the latter (Artinian ring, as a finite $\big(V((\mathbf{X}))\otimes_VW\big)$-vector space) is a direct product of fields (note that $V((\mathbf{X}))\otimes_VW$ is a field because $V\rightarrow V((\mathbf{X}))$ is a primary field extension as   $V$ is algebraically closed in $V((\mathbf{X}))$) \footnote{Here, the reader might find the following mathoverflow question helpful: \href{https://mathoverflow.net/questions/82083/when-is-the-tensor-product-of-two-fields-a-field}{When is the tensor product of two fields a field?}.}.

	
	As $R\otimes_VW$ admits a (generically)  separable Noether  normalization  and it is an excellent normal henselian   local ring with an algebraically closed residue field (Lemma \ref{LemmaExcellentProperty}), we can apply Proposition \ref{PropositionCompleteInfiniteResidueField}\ref{itm:2Henselian} to obtain a generically separable quasi-Gorenstein finite cover $S'$ of $R\otimes_VW$.
    Such a quasi-Gorenstein finite cover $S'$ of  $R\otimes_VW$  is in fact ascended  from  a module finite   $R\otimes_VV'_i$-algebra $S$, where $V'_i$ is a module finite \'etale local  $V$-algebra, as in the statement of Lemma \ref{LemmaStrictHenselizationAndDVRS}:
    
    Namely, the coefficients of the polynomials $f_i(\mathbf{Y})$  generating the defining ideal of $S'$, in a presentation $$S'=(R\otimes_VW)[\mathbf{Y}]/\big(f_i(\mathbf{Y})\big)$$    
    as well as the coefficients of some  monic polynomials $h_j(Z)$ (respectively, polynomials   $g_{j,i}(\mathbf{Y})$) corresponding to  each $Y_j$ and describing the integrality of the image of $Y_j$ in $S'$ over $R\otimes_V W$ by $$ h_j(Y_j)=\sum g_{j,i}(\mathbf{Y})f_i(\mathbf{Y}),\ \ \text{s.t. $h_j(Z)\in (R\otimes_V W)[Z]$ is a monic poly.,}$$ belong to some $R\otimes_VV'_i$, for some module finite \'etale local $V$-algebra $V_i$ as in the statement of Lemma \ref{LemmaStrictHenselizationAndDVRS}. This enables  us to consider $$S=(R\otimes_VV'_i)[\mathbf{Y}]/\big(f_i(\mathbf{Y})\big)$$ as a module finite $R\otimes_VV'_i$-algebra which satisfies clearly $$S\otimes_{R\otimes_VV'_i}(R\otimes_VW)=S'.$$
    
    As $S$ is finite over the henselian local ring $R\otimes_VV'_i$, so $S$ is a finite product of local rings and since its faithfully flat base change as written above is the connected space $\Spec(S')$, so $S$ has to be a local ring. A similar reasoning shows that $S$ is normal and the quasi-Gorensteinness of $S$ also follows from \cite[Theorem 4.1]{AoyamaGotoOnTheEndomorphism}, as $S\rightarrow S'$, as a base change of the  formally  \'etale morphism $R\otimes_VV'_i\rightarrow R\otimes_VW$, is formally \'etale and thus a regular local homomorphism (\cite[Tag 07PM]{StacksProject}).	
	Thus the statement follows from the observation that  $S$ is a module finite local extension of $R$ as so is $R\otimes_VV'_i$.

 However, we still may need  to discuss the separability of $\Frac(R)\rightarrow \Frac(S)$. But, as we have argued above $\Frac(R\otimes_VW)$  is a localization of $\Frac(R)\otimes_VW$ (thus formally \'etale by \cite[Tag 04EG]{StacksProject}) and the base change $\Frac(R)\rightarrow \Frac(R)\otimes_VW$ is formally \'etale as well by Lemma \ref{LemmaExcellentProperty} and \cite[Tag 02HJ]{StacksProject}, therefore $\Frac(R)\rightarrow  \Frac(R\otimes_V W)$ is formally \'etale (\cite[Tag 02HL and Tag  02HI]{StacksProject}. Then, so is its composition with the \'etale (separable) field extension $\Frac(R\otimes_VW)\rightarrow \Frac(S')$ (\cite[Tag 00UR and Tag 00U3]{StacksProject}), consequently $\Frac(R)\rightarrow \Frac(S')$ is formally \'etale. Finally, its subextension $\Frac(R)\rightarrow \Frac(S)$ is \'etale as well, i.e. separable (this can be seen for example using  Jacobi-Zariski exact sequence (\cite[Tag 00S2]{StacksProject}) as well as \cite[Tag 00U2(7), Tag 02H9 and Tag 031J(6)]{StacksProject}).
\end{proof}


\section{Quasi-Gorenstein finite cover: the projective variety case}

The aim of this section is to prove the projective variety analogue of Theorem \ref{TheoremCoverCompleteCase} (see Theorem \ref{TheoremQuasiGorensteinCoverProjectivevariety}). We split the proof into the following preparatory lemmas. Among them, the key one is Lemma \ref{LemmaProjectiveVarietyQuasiGorensteinCoverFromQGorenstein}. Throughout this section $K_Z$ (respectively, $K(Z)$) denotes the canonical divisor   (respectively, the function field) of a projective variety $Z$ over a field $k$. Note that the canonical sheaf, $\omega_{Z/k}=\mathcal{O}_Z(K_Z)$, is independent of the base field $k$ in the sense that $\omega_{Z/k'}=\omega_{Z/k}$ for a finite field extension $k'/k$ and a $k'$-projective variety $Z$ (see \cite[Lemma 2.7]{TanakaBehavior}). In this article, by a projective variety over a field $k$ we mean a projective integral (irreducible and reduced) scheme of finite type over $k$.  In particular, in the statement of the    lemma below we used the term ``projective scheme" to let the scheme $X$ therein to be non-reduced.

\begin{lemma}\label{LemmaProjectiveVarietyQGorensteinExtension} Let $X$ be an irreducible projective scheme over  a  perfect field $k$. Then, there exists a finite surjective morphism $f:Y\rightarrow X$ from a $\mathbb{Q}$-Gorenstein projective variety $Y$ to $X$.
\end{lemma}
\begin{proof}
  It suffices to prove the lemma for the normalization of the reduction, $X_\text{red}$, of $X$ (\cite[Tag 035B, Tag 035S and Tag 035Q(2)]{StacksProject}). Thus, we may and we do assume that $X$ is a normal projective variety. Let     $\delta:X\rightarrow \mathbb{P}^d_k$ be a generically separable Noether normalization of $X$.  In characteristic zero, simply take Proj of a Noether normalization of a standard coordinate ring of $X$ obtained by a homogeneous system of parameters consisting of elements of degree $1$ (see \cite[Theorem 1.5.17]{BrunsHerzogCohenMacaulay}). In the prime characteristic case it exists by \cite[V, Section 4, Theorem 8]{ZariskiSamuelCommutative} or by \cite[Theorem 1]{KedlayaMoreEtale}, noticing that $X$ is geometrically reduced by the perfectness of $k$ \footnote{It is obvious that the existence of a separable Noether normalization implies that the function field should be separably generated over $k$. It is interesting that this condition is  superfluous in the complete local case! Moreover, in contrast to the first reference, the second reference does not necessitate the infinity of $k$.}. As in the proof of Proposition \ref{PropositionCompleteInfiniteResidueField}, the  (separable)  function fields extension $[K(X):K(\mathbb{P}_k^d)]$ is a subextension of a finite Galois extension $[L:K(\mathbb{P}_k^d)]$. 
  Let $\{U_i\}$ be the affine open cover of $X$ given by $U_i=\delta^{-1}\big(D_+(x_i)\big)$  for the affine cover $D_+(x_i)$ of $\mathbb{P}^d_k=\Proj(k[x_0,\ldots,x_d])$.  
  
   Let $g:\Spec(L)\rightarrow X$ be the morphism given by the generic point of $X$, and let $\mathcal{A}$ be the integral closure of $\mathcal{O}_X$ inside $g_*(\mathcal{O}_{\text{Spec}(L)})$, see \cite[Tag 035G]{StacksProject} for the definition (the $\mathcal{O}_X$-module $g_*(\mathcal{O}_{\text{Spec}(L)})$ is quasi-coherent by \cite[II, Proposition 5.8(c)]{HartshorneAlgebraic}).  Set  $Y:=\Spec_X(\mathcal{A})$ that is equipped with its affine morphism $f:Y\rightarrow X$ such that $f^{-1}(U_i)=\Spec\big(\overline{\mathcal{O}_X(U_i)^L}\big)$. In particular, $f$ is a finite morphism (see \cite[Tag 01WI]{StacksProject} and \cite[Tag 032L]{StacksProject}),  and thus $Y$ is projective over $k$ by \cite[Tag 0B5V and Tag 0B45(5)]{StacksProject}.  
   The ample invertible sheaf  $\mathcal{L}:=(\delta \circ f)^*\big(\mathcal{O}_{\mathbb{P}^d_k}(1)\big)$ on $Y$, provides the module-finite graded ring homomorphism of (normal) coordinate rings $R\big(\mathbb{P}^d_k,\mathcal{O}_{\mathbb{P}^d_k}(1)\big)\rightarrow R(Y,\mathcal{L})$ whose induced   fraction fields extension is the  extension, $[L(t):K(\mathbb{P}^d_k)(t)]$, obtained from the Galois extension  $[L:K(\mathbb{P}^d_k)]$ by adjoining an indeterminate $t$, and thus it  is also a Galois finite field extension.
   
   Here, thus, a (local) graded version of \cite[Theorem 1.3(2) and Corollary 2.4]{KuranoOnRoberts} holds, implying that  $R(Y,\mathcal{L})$ is $\mathbb{Q}$-Gorenstein,  thus $$\omega_{R(Y,\mathcal{L})}^{(n)}=\Gamma_*\big(Y,\mathcal{O}_Y(K_Y)\big)^{(n)}=\Gamma_*\big(Y,\mathcal{O}_Y(nK_Y)\big)\cong R(Y,\mathcal{L})(a)$$ for some $a\in\mathbb{Z},$ from which we get that $\mathcal{O}_Y(nK_Y)\cong \mathcal{L}^a$ is an invertible sheaf (here $\Gamma_*(X,\mathcal{F})$ denotes the associated graded module to a coherent sheaf $\mathcal{F}$ on $R(X,\mathcal{L})$). Thus  $Y$ is $\mathbb{Q}$-Gorenstein and the proof is complete.
\end{proof}

The first part of the next lemma is comparable to   \cite[Theorem 2.9]{TanakaBehavior} which  applies to the more general case of a separated scheme with the cost of requiring  Cohen-Macaulayness.

\begin{lemma}\label{LemmaCanonicalSheafAndBaseChange}
  Let $X$ be a projective scheme over a field $k$ and $k'/k$ be a field extension. Let $\epsilon:X_{k'}:=X\times_{k}k'\rightarrow X$ be the projection. Then
  \begin{enumerate}[(i)]
    \item \label{itm:ItemLemmaCanonicalSheafAndFieldBaseChange}$\omega_{X_{k'}/k'}=\epsilon^*(\omega_{X/k})$.
    \item \label{itm:ItemLemmaSymbolicPowerOfCanonicalSheafandBaseChange} If $X,X_{k'}$ are normal (irreducible) projective varieties, then $\omega_{X_{k'}/k'}^{(n)}=\epsilon^*(\omega_{X/k}^{(n)})$. Thus if $X$ is $\mathbb{Q}$-Gorenstein, then so is $X_{k'}$.
  \end{enumerate}
\end{lemma}
\begin{proof}
   \ref{itm:ItemLemmaCanonicalSheafAndFieldBaseChange} We consider $X$  as a closed subscheme of $\mathbb{P}^n_k$ of codimension $r$  (for some $n$).  Similarly, $X_{k'}$ is a closed subscheme of $\mathbb{P}^n_{k'}$ of codimension $r$ whose defining ideal sheaf is given by the pullback of the ideal sheaf of $X$ (\cite[Tag 01JU(1)]{StacksProject}) \footnote{Here, we recall that a faithfully flat extension preserves  codimension.}.  Thus 
   \begin{equation}
   \label{EquationCanoincalSheafExtDescription}
   \omega_{X/k}=\mathscr{E}xt^r_{\mathbb{P}^n_k}(\mathcal{O}_X,\omega_{\mathbb{P}^n_{k}/k}) \ \ \text{and} \ \ \omega_{X_{k'}/k'}=\mathscr{E}xt^r_{\mathbb{P}^n_{k'}}(\mathcal{O}_{X_{k'}},\omega_{\mathbb{P}^n_{k'}/k'}),
   \end{equation}
   see \cite[III, the proof of Proposition 7.5]{HartshorneAlgebraic} or    \cite[Definition 5.3.5]{IshiiIntroduction2nd} (which the latter has an additional integrality assumption). Also, by \cite[Tag 01V0]{StacksProject} and \cite[II, Definition in page 180 and Exercise 5.16(e)]{HartshorneAlgebraic}  for the projection $\epsilon_P:\mathbb{P}^n_{k'}\rightarrow \mathbb{P}^n_k$ we have
   \begin{equation}
   \label{EquationPullbackOfCanonicalSheaf}
   \epsilon_P^*(\omega_{\mathbb{P}^n_k/k})=\epsilon_P^*(\wedge^n\Omega_{\mathbb{P}^n_k/k})=\wedge^n\epsilon_P^*(\Omega_{\mathbb{P}^n_k/k})=\wedge^n\Omega_{\mathbb{P}^n_{k'}/k'}=\omega_{\mathbb{P}^n_{k'}/k'}.
   \end{equation}

   Thus the statement follows from  displays (\ref{EquationCanoincalSheafExtDescription}), (\ref{EquationPullbackOfCanonicalSheaf}) as well as \cite[Proposition (12.3.4) and  (12.3.5)]{GrothendieckEGAIII}.
   

   \ref{itm:ItemLemmaSymbolicPowerOfCanonicalSheafandBaseChange} This part follows from the previous part as well as \cite[Proposition 5.2.8]{IshiiIntroduction2nd}.
\end{proof}

The following lemma is probably well-known to experts and is needed for its subsequent lemma, in which it plays an important role.

\begin{lemma}\label{DiagonalLemma}
Let $X$ be a  projective variety over a field $k$ with line bundles $M,N$. Let $\mathcal{I}_{\Delta}$ denote the ideal sheaf of the diagonal closed immersion $\Delta:X\rightarrow X\times X$. If $H^1\big(X\times X,\mathcal{I}_\Delta\otimes(M\boxtimes N)\big)=0$, then the multiplication map $H^0(X,M)\otimes H^0(X,N)\rightarrow H^0(X,M\otimes N)$ is surjective.
\end{lemma}
\begin{proof} 
  From the exact sequence $0\rightarrow \mathcal{I}_\Delta\rightarrow X\times X\rightarrow \Delta_*(\mathcal{O}_X)\rightarrow 0$, we get the cohomology exact sequence $H^0(X\times X,M\boxtimes N)\overset{\alpha}{\rightarrow} H^0\big(X\times X,\Delta_*(\mathcal{O}_X)\otimes(M\boxtimes N)\big)\rightarrow H^1\big(X\times X,\mathcal{I}_\Delta\otimes(M\boxtimes N)\big)$. Applying our hypothesis, thus the map $\alpha$ is surjective. Then, the statement follows from  K\"unneth formula $$H^0(X\times X,M\boxtimes N)\cong H^0(X,M)\otimes H^0(X,N),$$  because by \cite[Tag 04CI]{StacksProject}
  \begin{align*}
    H^0\big(X\times X,\Delta_*(\mathcal{O}_X)\otimes(M\boxtimes N)\big)&\cong 
    H^0\big(X\times X,\Delta_*(\Delta^*(M\boxtimes N)\big)
    \\ &\cong 
    H^0\big(X\times X,\Delta_*(M\otimes N)\big)
    \\ &\cong 
    H^0(X,M\otimes N).\qedhere
  \end{align*}  
\end{proof}

The proof of the following lemma is  using the idea of finding a way to force the lower symbolic powers to be generated in degree one, as well as  to work with the section ring of a divisor of the form $nK_X+H$ for some suitable  ample Cartier divisor $H$. This idea was suggested to the author  by Karl Schwede in conversations the author had with him and Peter McDonald.

\begin{lemma}\label{LemmaProjectiveVarietyQuasiGorensteinCoverFromQGorenstein}
Let   $X$ be a $\mathbb{Q}$-Gorenstein normal projective variety over an infinite field $k$. Then $X$ admits a quasi-Gorenstein finite cover.
\end{lemma}
\begin{proof}
  Fix some natural number $n$ so that   $nK_X$ is Cartier (which exists  in view of the definition of $\mathbb{Q}$-Gorensteinness). Let $d:=\dim(X)$ and $d'\ge d+1$ be the least multiple of $n$.  Let $H$ be a very ample Cartier divisor   on $X$. Fix  some natural number $v$   such that $
  \mathcal{O}_X(nK_X+wH)$ is a very ample line bundle for each $w\ge v$,  such that  
  \begin{equation}   
  \label{AppliedCohomologyVanishingOverX}
  \forall\ 1\le i\le d,\ 1\le s\le n-1,\ j\ge v: \ \ H^i\bigg(X,\mathcal{O}_X\Big(\big(n(d-i+1)-s\big)K_X+jH\Big)\bigg)=0 
  \end{equation}
  (here the index $s$ will play the role of the $s$-th symbolic power) and such that 
  \begin{equation}   
  \label{AppliedCohomologyVanishingOverXtimesX}
  \forall\    1\le j\le nd': \ \ H^1\Big(X\times X,\mathcal{I}_\Delta\otimes \big(\mathcal{O}_X(jK_X)\boxtimes \mathcal{O}_X(nK_X)\big)\otimes \big(\mathcal{O}_X(jH)\boxtimes \mathcal{O}_X(nH)\big)^{v}\Big)=0. 
  \end{equation}
  Such $v$ exists  by applying Serre's vanishing theorem to the relevant coherent sheaves with respect to the very ample $\mathcal{O}_X$-line bundle $\mathcal{O}_X(H)$ (respectively, $\mathcal{O}_{X\times X}$-very ample line bundle $\mathcal{O}_X(jH)\boxtimes \mathcal{O}_X(nH)$, $1\le j\le nd'$) for display (\ref{AppliedCohomologyVanishingOverX}) (respectively, for display (\ref{AppliedCohomologyVanishingOverXtimesX})), and then taking the maximum of the  resulted numbers.

   Set $D:=nK_X+nvH$, enabling us to choose $$R:=R(X,D)=\oplus_{j\ge 0}\Gamma\big(X,\mathcal{O}_X(jnK_X+jnvH)\big)$$ as the coordinate ring of $X$. The coherent sheaf  $\mathcal{F}:=\mathcal{O}_X(-K_X-vH)$  provides us with the $R$-module $$M:=\Gamma_*(X,\mathcal{F})=\oplus_{n\in\mathbb{Z}}\Gamma\big(X,\mathcal{F}\otimes\mathcal{O}_X(nD)\big)$$ whose $n$-th symbolic power $$M^{(n)}=\Gamma_*\big(X,\mathcal{O}_X(-nK_X-nvH)\big)\cong R(-1)$$ is a graded free module generated in degree $1$. Let $$S:=\mathscr{R}[Mt]=\oplus_{i\ge 0}M^{(i)}t$$ be the symbolic Rees algebra of the $R$-module  $M$. Using the vanishings in  displays (\ref{AppliedCohomologyVanishingOverX}) and (\ref{AppliedCohomologyVanishingOverXtimesX}), we aim to prove the following claim. 
   
   \textbf{Claim:} The graded module $M$ and its   symbolic powers $M^{(s)}$  ($2\le s\le n-1$)  are also generated in degree $1$. 
 
   Proof of the claim. We prove the claim for some fixed $1\le s\le n-1$. As the result of vanishings in (\ref{AppliedCohomologyVanishingOverX}),   the coherent sheaf $\mathcal{O}_X(-sK_X-svH)$ is $(d+1)$-regular (and thus $d'$-regular by \cite[Theorem 1.8.5(iii)]{LazarsfeldPositivityI}) with respect to $\mathcal{O}_X(nK_X+nvH)$, because $$H^i\bigg(X,\mathcal{O}_X\Big(\big(n(d+1-i)-s\big)K_X+ \underset{\ge v}{\underbrace{v'}}H\Big)\bigg)=0.$$ From this and \cite[Theorem 1.8.5(ii)]{LazarsfeldPositivityI}, we get  the surjectivity of the multiplication maps
   \begin{equation} 
   \label{EquationGeneration1}
     H^0\Big(X,\mathcal{O}_X\big((d'n-s)(K_X+vH)\big)\Big)\otimes H^0\Big(X,\mathcal{O}_X\big(\ell(nK_X+nvH)\big)\Big)\rightarrow H^0\bigg(X,\mathcal{O}_X\Big(\big((d'+\ell)n-s\big)(K_X+vH)\Big)\bigg).
   \end{equation}

  Moreover,  display  (\ref{AppliedCohomologyVanishingOverXtimesX}) and Lemma \ref{DiagonalLemma} yield the surjectivity of the multiplication maps
  \begin{equation}
  \label{EquationGenration2}
  H^0\big(X,\mathcal{O}_X(jK_X+jvH)\big)\otimes H^0(X,\mathcal{O}_X(nK_X+nvH)\big)\rightarrow H^0\Big(X,\mathcal{O}_X\big((j+n)(K_X+vH)\big)\Big)
  \end{equation}
  for each $1\le j\le nd'$. From the surjections in the displays (\ref{EquationGeneration1}) and (\ref{EquationGenration2}) it is readily seen that our claim holds. Thus, the proof of our claim is complete.
 
 We endow the symbolic Rees algebra $S$ with an $\mathbb{N}_0$-graded structure such that
 $$S_{[i]}:=\oplus_{j\in \mathbb{N}_0} M^{(j)}_{[i]}t^j.$$
 Our aforementioned claim shows that $S$ equipped with this grading is a standard graded ring. Additionally, applying our aforementioned claim  and considering a set of  degree one generators $\{w_{j,i}\}$ where $1\le j\le n-1$ (respectively, $\{m_1,\ldots,m_{\mu}\}$ and $\{f\}$) for $M^{(j)}$   (respectively, for $R_+$ and $M^{(n)}$),  by the same arguments as in the proof of Proposition \ref{PropositionCompleteInfiniteResidueField} we can apply Bertini theorem   to conclude the existence of a general choice of an element $z$ of the form 
	     \begin{equation}
	     	\label{EquationTheBertinieElementGraded}
	       z=ft^n+(\sum_i k_{n-1,i}w_{n-1,i})t^{n-1}+\cdots + (\sum_i k_{1,i}w_{1,i})t + (\sum_i k_{0,i} m_i)
	     \end{equation}  
         such that $S/(z)$ is normal and is a module-finite extension of $R$. Moreover, such  $z$ is indeed a degree one homogeneous element  with respect to the $\mathbb{N}_0$-grading of $S$  defined above, and $R\rightarrow S/(z)$ is a module-finite homogeneous homomorphism of standard graded rings inducing a finite morphism of normal varieties $Y:=\Proj\big(S/(z)\big)\rightarrow \Proj(R)=X$. It remains thus only to observe that $Y$ is quasi-Gorenstein, which we shall discuss it as follows. 

         Applying again \cite[Theorem (4.5)(1)]{GotoEtALOnTheStructure}, this time  with respect to the symbolic Rees algebra $S$ of $R$, we get the (canonical divisor) realization 
         \begin{align}
           \label{EquationCanonicalDivisorFormula}
           K_{S}=\kappa^*\Big([\Gamma_*\big(X,\mathcal{O}_X(K_X)\big)]\Big)+\kappa^*\Big([\Gamma_*\big(X,\mathcal{O}_X(-K_X-vH)\big)]\Big)
         =\kappa^*\Big([\Gamma_*\big(X,\mathcal{O}_X(-vH)\big)]\Big)
         \end{align}
         where $\kappa^*:\text{cl}(R)\rightarrow \text{cl}(S)$ is the  map on the divisor class groups induced by $\kappa:\Spec(S)\rightarrow \Spec(R)$ (which is an  isomorphism of Abelian groups, see \cite[Proposition (4.4)(1)]{GotoEtALOnTheStructure}), and the notation $[N]$ denotes the Weil divisor class corresponding to a rank one reflexive sheaf $N$.   One point is that,  $\Gamma_*\big(X,\mathcal{O}_X(-vH)\big)$ is an invertible sheaf on $\Spec(R)\backslash \{R_+\}$ (because $\mathcal{O}_X(vH)$ is an invertible sheaf of $\mathcal{O}_X$-modules). Thus, from this fact together with (\ref{EquationCanonicalDivisorFormula}) we conclude that  $K_{S}\rceil_{\kappa^{-1}\big(\Spec(R)\backslash \{R_+\}\big)}$ is a Cartier divisor.   This shows that $\kappa^{-1}\big(\Spec(R)\backslash \{R_+\}\big)$ is quasi-Gorenstein.  On the other hand, since $R\rightarrow S/(z)$ is finite, so the image of the  punctured spectrum  $\text{Spec}^\circ\big(S/(z)\big)$,  under the closed immersion $\text{Spec}\big(S/(z)\big)\rightarrow \Spec(S)$, lies in $\kappa^{-1}\big(\Spec(R)\backslash \{R_+\}\big)$. It turns out that  $\text{Spec}^\circ\big(S/(z)\big)$  is the closed hypersurface subscheme of $\kappa^{-1}\big(\Spec(R)\backslash \{R_+\}\big)$,   defined by the principal ideal sheaf of $z$. Hence, since $\kappa^{-1}\big(\Spec(R)\backslash \{R_+\}\big)$ is quasi-Gorenstein and its hypersurface $\text{Spec}^\circ\big(S/(z)\big)$ is normal, again from \cite[Corollary 2.8]{TavanfarTousiAStudy} we can conclude that $\text{Spec}^\circ\big(S/(z)\big)$ is quasi-Gorenstein. Finally,  this shows that $Y=\text{Proj}\big(S/(z)\big)$ is quasi-Gorenstein as well, as needed. Thus the proof is complete.
\end{proof}

We give the details of the proof of the following descent lemma, for the sake of completeness or a potential convenience.

\begin{lemma}\label{LemmaProjectiveDescent} Let $L/k$ be an algebraic field extension and $X$ be a normal projective variety over $k$ such that $X_L:=X\times_kL$ is irreducible. Let $X^v_L$ be the normalization of $X_L$ given by  the integral closure of the reduction $(X_L)_{\text{red}}$, of $X_L$, at its generic point. If $X^v_L$ admits a quasi-Gorenstein finite cover, then so does $X$.
\end{lemma}
\begin{proof}
   Let $f:Y\rightarrow X_L^{v}$ be a quasi-Gorenstein finite cover for $X_L^{v}$ and let $\pi:X_L^v\rightarrow X$ be the composition of the  projection and the normalization. Then $\pi\circ f:Y\rightarrow X$ is a surjective morphism of integral schemes that is also an integral morphism. Also, $\varphi:R\big(X,\mathcal{O}_X(1)\big)\rightarrow R\Big(Y,(\pi\circ f)^*\big(\mathcal{O}_X(1)\big)\Big)$ is an integral ring homomorphism. Let $L':=H^0(Y,\mathcal{O}_Y)$. There is a finite extension $F$ of $k$, contained in $L'$, such that it contains all of the coefficients of 
   \begin{itemize}
       \item the polynomials of the defining ideal of $R\Big(Y,(\pi\circ f)^*\big(\mathcal{O}_X(1)\big)\Big)$ as a quotient of a polynomial ring over $L'$.
       \item the polynomials representing   the image, under $\varphi$, of the generators of $R\big(X,\mathcal{O}_X(1)\big)$ over $k$.
       \item the polynomials representing the integrality of the generators of $R\Big(Y,(\pi\circ f)^*\big(\mathcal{O}_X(1)\big)\Big)$ over $R\big(X,\mathcal{O}_X(1)\big)$ similarly as in the case of the polynomials $h_j,g_{j,i}$ in the proof of Theorem \ref{TheoremCoverCompleteCase}.
   \end{itemize} 
   Then, we obtain a module-finite graded extension $A$ of $R\big(X,\mathcal{O}_X(1)\big)$, over $F$, such that $A\otimes_{F}L'\cong R\Big(Y,(\pi\circ f)^*\big(\mathcal{O}_X(1)\big)\Big)$ and therefore $Y\cong \Proj(A\otimes_FL')\cong \Proj(A)\times_FL'$. Then from Lemma \ref{LemmaCanonicalSheafAndBaseChange}\ref{itm:ItemLemmaCanonicalSheafAndFieldBaseChange}, \cite[Tag 0098]{StacksProject}, the faithfully flatness of the base field change and  quasi-Gorensteinness and normality of $Y$ we conclude that $\Proj(A)$ is a quasi-Gorenstein normal projective variety over $F$. Finally, the injective module-finite graded ring homomorphism $R\big(X,\mathcal{O}_X(1)\big)\rightarrow A$ yields a finite surjective morphism $\Proj(A)\rightarrow X$ that is a quasi-Gorenstein finite cover as needed.
\end{proof}

 The next result, which is the main result of this section,  establishes the existence of quasi-Gorenstein finite covers for projective varieties and covers the remained  case of characteristic $2$ where Theorem \ref{TheoremKawamta} does  not apply.  
\begin{theorem}\label{TheoremQuasiGorensteinCoverProjectivevariety} Let $X$ be a normal projective variety over a field. Then $X$ admits a quasi-Gorenstein finite cover.
\end{theorem}
\begin{proof}
   Suppose that $K$ is finite, thus perfect. Then $X$ admits a $\mathbb{Q}$-Gorenstein finite cover $Y$ by Lemma \ref{LemmaProjectiveVarietyQGorensteinExtension}. By the perfectness of $k$ as well as \cite[Tag 038O]{StacksProject}, $Y=\Proj\Big(\oplus_{n\ge 0}\Gamma\big(Y,\mathcal{O}_Y(n)\big)\Big)$, is geometrically normal. Therefore setting $k':=H^0(Y,\mathcal{O}_Y)$, the $k'$-projective variety $Y$  is geometrically normal and geometrically integral over $k'$ by \cite[Tag 0FD3]{StacksProject}. Consequently, $Y_{k^{\text{alg}}}:=Y\times_{k'}k^{\text{alg}}$ is a normal (irreducible) projective variety over the algebraic closure $k^{\text{alg}}$ of $k$ (or $k'$). Moreover,   the (field) base change $Y_{k^{\text{alg}}}$, of $Y$, is $\mathbb{Q}$-Gorenstein as well (Lemma \ref{LemmaCanonicalSheafAndBaseChange}). Thus  $Y_{k^{\text{alg}}}$ admits a quasi-Gorenstein finite cover $Z$ by Lemma \ref{LemmaProjectiveVarietyQuasiGorensteinCoverFromQGorenstein}, and hence so does $Y$  by Lemma \ref{LemmaProjectiveDescent}.   Consequently, so does $X$ and the statement follows for the case where $k$ is finite.

Now, we suppose that $k$ is an infinite  field whose perfect closure is denoted by $k^{\infty}$. In view of \cite[Tag 01OM and Tag 0BRA]{StacksProject}, $X_{k^\infty}:=X\times_kk^{\infty}$ is a, possibly non-reduced,  irreducible projective scheme over $k^\infty$. Thus, it admits a finite surjective morphism $f:Y\rightarrow X_{k^\infty}$ from a normal quasi-Gorenstein projective $k^\infty$-variety $Y$  by Lemma \ref{LemmaProjectiveVarietyQGorensteinExtension} and Lemma \ref{LemmaProjectiveVarietyQuasiGorensteinCoverFromQGorenstein}. Following the notation of the statement of Lemma \ref{LemmaProjectiveDescent}, such a quasi-Gorenstein finite cover $f$ necessarily factors through $Y\rightarrow X^v_{k^\infty}\rightarrow X_{k^\infty}$ (\cite[Tag 035Q(4)]{StacksProject}), implying that $X^v_{k^\infty}$ admits a quasi-Gorenstein finite cover. So the statement follows from Lemma \ref{LemmaProjectiveDescent} in this infinite case as well and the proof is complete.
\end{proof}

\begin{remark} As in the case of complete local normal domains (Theorem \ref{TheoremCoverCompleteCase}), one can probably find a   quasi-Gorenstein finite cover of an arbitrary normal projective variety whose induced extension of function fields is separable. For a proof, one can argue as in the complete local case.
\end{remark}

The example in the next remark is pointed out to the author by Karl Schwede.
\begin{remark}  In contrast to the   local (affine) case, projective varieties do not necessarily admit a \textit{Calabi-Yau} finite cover (i.e. a finite cover whose canonical divisor is trivial)  although they admit a quasi-Gorenstein finite cover  by Theorem \ref{TheoremKawamta} and Theorem \ref{TheoremQuasiGorensteinCoverProjectivevariety} (which the latter covers the case of characteristic $2$ as well). For example, let $Y$ be a nonsingular curve   over $\mathbb{C}$ and suppose  that there exists a finite Calabi-Yau cover $f:X\rightarrow Y$, thus $g(X)=\ell\big(\Gamma(X,\omega_X)\big)=\ell\big(\Gamma(X,\mathcal{O}_X)\big)=1$. Then, from \cite[IV, Corollary 2.4]{HartshorneAlgebraic} we get  necessarily that $g(Y)=1$ (and $Y$ is an elliptic curve) because the ramification divisor  $R$ of $f$ is indeed an effective divisor. Consequently, a non-elliptic curve admits no Calabi-Yau finite cover.
\end{remark}


\section{An application of a variation of  Theorem \ref{TheoremKawamta}}
 In \cite{RobertsAbelian}, Paul Roberts actually proves that:
 
  An Abelian extension of a local factorial domain $R$  whose Galois group (of the  extension of  the fraction fields) has order relatively prime to the characteristic of the residue field of $R$, is necessarily free as an $R$-module, so-called \textit{module-free} \footnote{Therefore  if $R$ is regular, then such an Abelian extension is necessarily Cohen-Macaulay.}.

  The following result, which is an immediate corollary to a variation of Theorem \ref{TheoremKawamta}, provides a converse to the above  Robert's result:
  
  \begin{corollary}\label{CorollaryConverseToRoberts}  Let  $R$ be a  normal domain of residual characteristic not equal to $2$ which is either   $\mathbb{N}_0$-graded over an infinite field $R_{[0]}$ or a  complete local ring with an infinite residue field. Suppose that  all of the  rank $2$  module-finite normal domain extensions of $R$, or more restrictively all of the module-finite normal domain generically square root extensions  of $R$, are free as an $R$-module. Then, $R$ is a factorial domain.
  \end{corollary}
\begin{proof}
  The proof of this fact is as easily as to apply a revised version of  Theorem \ref{TheoremKawamta} by replacing the canonical divisor $K_X$ in its statement with an arbitrary Weil divisor, and noting that Bertini theorem in its statement is available for both of the complete local case and the graded case (the existence of a reduced (effective)  Weil divisor $D$ as in Theorem \ref{TheoremKawamta} follows from the same blow-up  method used in the proof of \cite[Lemma 3.5]{McDonaldMultiplier}, see also    \cite[Theorem 2.17]{BhatManyEtALGlobally}).  
  
  This revised version of Theorem \ref{TheoremKawamta}  implies that for any pure height one ideal $\mathfrak{a}$ of $R$, the $R$-module $R\oplus \fa$ acquires a  normal domain $R$-algebra structure (inducing a square root extension of the fractions fields) which has to be module-free by our hypothesis. This implies that $\fa$ is free. Therefore, $R$ is a factorial domain. 
\end{proof}  

  The above fact implies an immediate important application as given in the next corollary. This corollary extends the main result of \cite{PatankarVanishing} to the ($2$-dimensional) complete local case. 

 \begin{corollary}\label{CorollaryExtensionOfPatankar'sResult} Let  $(R,\fm,k)$ be a $2$-dimensional complete local normal domain of equal characteristic zero. Let $R^+$ be the absolute integral closure of $R$. If  $\text{Tor}^R_{i}(R^+,k)=0$  for some $i\ge 1$, then $R$ is a regular local ring.
 \end{corollary}
 \begin{proof}
   By  \cite[Proposition 4.2]{PatankarVanishing}, $R$ is normal.  By some arguments in the proof of \cite[Theorem A]{PatankarVanishing}, any module finite normal domain extension of $R$ is module-free. 
   
   Let $k$ be the coefficient field of $R$. In view of Lemma \ref{LemmaExcellentProperty}, $\big(R\otimes_k \overline{k},(\fm)\big)$ is an excellent normal henselian local domain ($\overline{k}$ denotes the algebraic closure of $k$) which is a direct limit of module-finite complete local normal domains of the form $R\otimes_kL$ wherein $L$ is a finite algebraic extension of $k$. 
   
   Let $T$ be a finite normal domain extension of $R\otimes_k\overline{k}$. Thus, $T$   is of the form $T=S\otimes_{(R\otimes_kL)}(R\otimes_k\overline{k})$ for some finite algebraic extension $L/k$ and some normal domain module-finite  extension $S$ of $R\otimes_kL$ (as argued in the fifth paragraph of the proof of Theorem \ref{TheoremCoverCompleteCase}).  Thus $S$ is also a normal domain module-finite extension  of $R$, therefore it is a free $R$-module in view of  the first paragraph of the proof. Hence, $S$ is a free $R\otimes_kL$-module as well by virtue of the local criterion of flatness (\cite[Tag 00ML]{StacksProject}). Consequently, $T$ is a free $R\otimes_k\overline{k}$-module. Since, $T$ was arbitrary, we conclude from Corollary \ref{CorollaryConverseToRoberts} that $R\otimes_k\overline{k}$ is a factorial. 

   Then,  we can invoke to \cite[Theorem (17.4)]{LipmanRational} 
   to conclude that $R\otimes_k\overline{k}$ has at most rational singularities. Therefore, the proof is complete as proved in \cite[Theorem 4.5]{PatankarVanishing} (because $R\otimes_k\overline{k}$ is a Gorenstein rational isolated singularity of dimension $2$, thus admits a non-singular finite   cover (\cite{PostralEtALCovers}) which should be module-free as discussed before).
 \end{proof}
 
The next remark  suggests that the use of almost mathematics in the proof of \cite[Theorem A]{PatankarVanishing} seems to be necessary.
\begin{remark} 
  A graded version of the proof of  Corollary \ref{CorollaryExtensionOfPatankar'sResult}  does not prove (the graded result)  \cite[Theorem A]{PatankarVanishing}.   This is because, there exists an $\mathbb{N}_0$-graded  $2$-dimensional factorial domain over an algebraically closed field $R_{[0]}$ which does not have a rational singularity. In \cite[Theorem 5.1]{MoriGradedFactorial} \footnote{Shigefumi Mori's result is stated for the almost geometric graded case, for the (more) general graded case see  \cite[Remark after Corollary (1.7)]{WatanabeSomeRemarks}},  such graded factorial domains are classified by which one can observe that there are such graded rings with non-negative $a$-invariant, thus with a non-rational singularity  (see \cite[Theorem (3.3)]{WatanabeSomeRemarks}). More precisely, by virtue of \cite[Example 3.2(case d=2)]{SinghSpiroffDivisor} the graded ring $\mathbb{C}[x,y,z]/(x^2+y^3+z^7)$ is factorial with a   positive $a$-invariant, thus a  non-rational   singularity. 
\end{remark}

\begin{remark} One year before getting aware of Theorem \ref{TheoremKawamta} and based on some previous observations,   the author had in mind the idea  that the statement of Corollary \ref{CorollaryConverseToRoberts} should be true (and by which one obtains   Corollary \ref{CorollaryExtensionOfPatankar'sResult}). The author expects that there should be an alternative (perhaps more conceptual) proof for concluding the factoriality of (compete, graded, etc.) normal domains whose Abelian extensions are all module-free, specially a proof which also covers the, yet unsettled,  residual characteristic $2$ case as well.  
\end{remark}

\section{Acknowledgment}
 The author thanks greatly Karl Schwede and Peter McDonald for many in-depth
conversations. The author also thanks Anurag K. Singh for answering our questions.

\end{document}